\documentclass[11pt,reqno]{amsart}
\usepackage[T1]{fontenc}
\usepackage{lmodern}
\usepackage{amsmath,amssymb,amscd,mathrsfs}
\usepackage[margin=1in]{geometry}
\usepackage{booktabs}
\usepackage{enumitem}
\usepackage{microtype}
\usepackage[colorlinks=true,citecolor=blue,linkcolor=blue,urlcolor=blue]{hyperref}
\allowdisplaybreaks

\newtheorem{theorem}{Theorem}[section]
\newtheorem{lemma}[theorem]{Lemma}
\newtheorem{proposition}[theorem]{Proposition}
\newtheorem{corollary}[theorem]{Corollary}
\theoremstyle{definition}

\theoremstyle{remark}

\newcommand{\bR}{\mathbb R}
\newcommand{\bZ}{\mathbb Z}
\newcommand{\F}{F}
\newcommand{\f}{f}
\newcommand{\aff}{\operatorname{aff}}
\newcommand{\spanR}{\operatorname{span}}
\newcommand{\dimn}{\operatorname{dim}}
\newcommand{\norm}[1]{\left\lVert #1\right\rVert}

\title[Johnson optimality and uniqueness]{Johnson optimality and uniqueness for distance ratio $1:\sqrt{2}$}

\author{Bryan Bu}
\email{bryanbu@umass.edu}

\author{Chin-Yen Lee}
\email{109281001@cc.ncu.edu.tw}

\author{Yu-Hang Sheng}
\email{joy09272002@gmail.com}

\author{Meng-Tsung Tsai}
\email{mttsai@as.edu.tw}

\author{Hyunjun Yi}
\email{hjyi2027@gmail.com}

\author{Wei-Hsuan Yu}
\email{whyu@math.ncu.edu.tw}

\begin{document}
\begin{abstract}
We prove that a two-distance set in $\bR^d$ with distance ratio
$1:\sqrt{2}$ has at most $d(d+1)/2$ points for every $d\ge9$.
Equality in full affine dimension characterizes the Johnson construction;
for $d\ge10$, the full-dimensional condition is automatic.
The proof separates configurations according to the connectivity of the
short-distance graph. The connected case uses root lattices and
intersection theorems, while the disconnected case uses an orthogonal
decomposition into component direction spaces and the span of their
center differences. We also prove a relative bound for Euclidean
two-distance sets and show that equality forces sphericity.
Together with the Larman--Rogers--Seidel theorem, these results give
maximum cardinality $d(d+1)/2$ for $9\le d\le21$ and $275$ for $d=22$,
without exhaustive enumeration.
\end{abstract}
\maketitle

\section{Introduction and notation}\label{sec:intro}

A finite set in Euclidean space is called a two-distance set if
the distances between distinct points take exactly two values.
A basic problem is to determine the largest possible cardinality
of such a set in $\bR^d$. Write
\begin{align*}
  M(d):=\max\{|X|:X\subset\bR^d
  \text{ is a finite two-distance set}\}.
\end{align*}

The theorem of Larman, Rogers, and Seidel
\cite{LarmanRogersSeidel} restricts the possible distance ratios
of sufficiently large two-distance sets.
If $|X|>2d+3$ and the two squared distances are $\alpha<\beta$,
then
\begin{align*}
  \frac{\alpha}{\beta}=\frac{\kappa-1}{\kappa},
  \quad
  2\le\kappa\le
  \left\lfloor\frac{1+\sqrt{2d}}2\right\rfloor,
\end{align*}
where $\kappa$ is an integer.
We call $\kappa$ the LRS parameter.
For each integer $\kappa\ge2$, let $M_\kappa(d)$ denote the
largest cardinality of a finite two-distance set in $\bR^d$ with
distance ratio $\sqrt{\kappa-1}:\sqrt{\kappa}$.
Neither definition requires full affine dimension.

Put $\F(r)=r(r+1)/2$ and $\f(r)=r(r+3)/2$.
The Johnson construction
\begin{align*}
  J(d+1,2)
  :=\{e_i+e_j:1\le i<j\le d+1\}
\end{align*}
lies in the $d$-dimensional affine hyperplane
$\{x\in\bR^{d+1}:\sum_i x_i=2\}$.
For $d\ge3$, it has affine dimension $d$, squared distances
$2$ and $4$, and $\F(d)$ points.
Thus $M(d)\ge M_2(d)\ge\F(d)$ for $d\ge3$.
For $d\ge9$, we have $\F(d)>2d+3$, so the LRS theorem applies
to every maximum two-distance set in $\bR^d$.

For $9\le d\le12$, the LRS theorem leaves only $\kappa=2$.
For $13\le d\le22$, it allows $\kappa=2$ or $\kappa=3$.
The relative bound developed in Section~\ref{sec:relative} gives
\begin{align*}
  M_3(d)\le\frac{24(d+1)}{24-d}
  \quad (d<24).
\end{align*}
For $9\le d\le21$, this bound is strictly smaller than $\F(d)$.
Consequently, every maximum two-distance set in these dimensions
has $\kappa=2$.

The situation changes in dimension $22$.
The Johnson construction has $\F(22)=253$ points, whereas the
McLaughlin construction gives a spherical two-distance set with
$275$ points and $\kappa=3$; see \cite[Example~4.10]{DelsarteGoethalsSeidel}.
The relative bound gives $M_3(22)\le276$.
We show that equality in the relative bound forces sphericity.
Since a spherical two-distance set in $\bR^{22}$ has at most
$\f(22)=275$ points, the value $276$ is impossible.
It follows that $M_3(22)=275$.
In dimension $23$, Ge, Koolen, and Munemasa~\cite{GeKoolenMunemasa}
constructed a two-distance set with $277$ points and distances
$2$ and $\sqrt{6}$.
This construction has LRS parameter $\kappa=3$ and yields
$M(23)\ge M_3(23)\ge277>\F(23)=276$.
Thus the Johnson construction is not optimal in dimension $23$
when the distance ratio is unrestricted.
Our main geometric result determines $M_2(d)$ in every dimension
$d\ge9$.

\begin{theorem}\label{thm:main}
Let $d\ge9$, and let $X\subset\bR^d$ be a two-distance set with
distance ratio $1:\sqrt2$. Then
\begin{align*}
  |X|\le\F(d)=\binom{d+1}{2}.
\end{align*}
Equality is attained by the Johnson construction.
If equality holds and $\dimn\aff(X)=d$, then $X$ is similar to
$J(d+1,2)$.
For $d\ge10$, equality automatically implies
$\dimn\aff(X)=d$, so the maximum configuration is unique up to
similarity.
\end{theorem}

The full affine-dimension assumption in dimension $9$ is necessary:
Lison\v{e}k's $45$-point configuration has affine dimension $8$
and can also be embedded in $\bR^9$; see \cite{Lisonek,BannaiEtAl}.
Combining Theorem~\ref{thm:main} with the LRS theorem and the
relative bound yields the following result.

\begin{corollary}\label{cor:maximum}
The maximum cardinality of a two-distance set in $\bR^d$ is
\begin{align*}
  M(d)=
  \begin{cases}
    \displaystyle\binom{d+1}{2}, & 9\le d\le21,\\
    275, & d=22.
  \end{cases}
\end{align*}
For $10\le d\le21$, every maximum configuration is similar to
the Johnson construction.
\end{corollary}

Table~\ref{tab:maxima} records these values and the LRS parameter
of a maximum configuration.
In dimension $22$, the spherical $275$-point configuration is
unique up to similarity; see \cite[Remark~4]{Musin}.
The classification of all Euclidean configurations attaining
$275$ points is not established here.

\begin{table}[ht]
\centering
\begin{tabular}{crrl}
\toprule
$d$ & $M(d)$ & Extremal $\kappa$
    & Uniqueness \\
\midrule
 9  &  45 & 2 & No \\
10 &  55 & 2 & Yes \\
11  &  66 & 2 & Yes \\
12  &  78 & 2 & Yes \\
13  &  91 & 2 & Yes \\
14 & 105 & 2 & Yes \\
15  & 120 & 2 & Yes \\
16  & 136 & 2 & Yes \\
17 & 153 & 2 & Yes \\
18 & 171 & 2 & Yes \\
19 & 190 & 2 & Yes \\
20 & 210 & 2 & Yes \\
21 & 231 & 2 & Yes \\
22 & 275 & 3 & Not established here \\
\bottomrule
\end{tabular}
\caption{Maximum cardinalities and uniqueness for two-distance
sets in $\bR^d$. Uniqueness is understood up to similarity and concerns configurations
attaining $M(d)$. For $10\le d\le21$, the unique extremal
configuration is $J(d+1,2)$. In dimension $9$, both $J(10,2)$
and Lison\v{e}k's configuration attain the maximum;
the latter has affine dimension $8$.}
\label{tab:maxima}
\end{table}

The spherical problem has been studied extensively.
Musin~\cite{Musin} proved that the maximum cardinality of a
spherical two-distance set in $\bR^d$ is $\F(d)$ for
$7\le d\le39$, except in dimensions $22$ and $23$.
Further bounds were obtained by Barg and Yu~\cite{BargYu}.
Glazyrin and Yu~\cite{GlazyrinYu} established the same spherical
maximum for $d\ge7$, with possible exceptions among dimensions
of the form $d=(2k+1)^2-3$.

Theorem~\ref{thm:main} determines the fixed-ratio maximum
$M_2(d)$ for every $d\ge9$ among arbitrary Euclidean
configurations and characterizes equality in full affine
dimension. It also gives uniqueness up to similarity for
$d\ge10$. Corollary~\ref{cor:maximum} determines the unrestricted
Euclidean maximum for $9\le d\le22$, extending the corresponding
spherical cardinality results to configurations that need not
lie on a sphere.

Computational enumeration and classification have also played
an important role, particularly in low dimensions; see
Lison\v{e}k~\cite{Lisonek}.
Throughout this paper, a maximum configuration means one
attaining the largest cardinality in the specified class.
This differs from an inclusion-maximal configuration, to which
no further point can be added while remaining in that class.
Our results concern maximum cardinality and the configurations
attaining it.

For $\kappa=2$, normalize the squared distances to $2$
and $4$, and associate a graph with $X$: its vertices are the
points of $X$, and two vertices are adjacent when their squared
distance is $2$. We call this the short-distance graph.
Our proof separates according to whether this graph is connected.

In the connected case, translating one point to the origin
places $X$ in an even integral lattice generated by vectors
of squared norm $2$. The classification of simply-laced root
systems reduces the problem to the irreducible cases of types
$A$ and $D$, together with a reducible case.
The first two cases are handled by intersection theorems for
finite sets. In the reducible case, an orthogonal decomposition
and spherical bounds give a strict upper bound below $\F(d)$.

In the disconnected case, points in different components are
all at squared distance $4$.
Each component lies on a sphere in its affine hull.
We show that the direction spaces of the components are
mutually orthogonal and are also perpendicular to the
differences between their centers.
This gives an orthogonal decomposition of the affine hull.
The resulting quadratic form on the center coefficients
reduces the cardinality estimate to three cases, according
to the signs of its parameters.
Each case gives $|X|<\F(d)$ when $\dimn\aff(X)=d\ge9$.

For $\kappa=3$, we use a relative bound depending on
the dimension and the distance ratio.
A spectral argument for the distance matrix proves this bound
and shows that equality forces the configuration to be spherical.
This supplies the remaining upper bounds needed for
Corollary~\ref{cor:maximum}.

The paper is organized as follows.
We first record the spherical estimates used throughout the proof.
Section~\ref{sec:connected} treats connected short-distance graphs,
and Section~\ref{sec:disconnected} treats disconnected ones.
Section~\ref{sec:relative} establishes the relative bound and
determines $M_3(22)$. Section~\ref{sec:completion} combines these
results to prove Theorem~\ref{thm:main} and
Corollary~\ref{cor:maximum}.

For a configuration $X$, write $\aff(X)$ for its affine hull.
We write $[n]=\{1,\ldots,n\}$, and use $\mathbf{1}$ and $J$
for the all-ones vector and matrix of the appropriate sizes.
If $X_i$ is a connected component of its short-distance graph,
put $V_i=\spanR(X_i-X_i)$ and $r_i=\dimn V_i$.
When considering subsets, we allow one of the two prescribed
distances to be absent.

\subsection{Spherical estimates}

\begin{lemma}\label{lem:sphere}
Let $Z$ lie on a sphere of positive squared radius $t$ in an $r$-dimensional Euclidean space, and suppose its pairwise squared distances belong to $\{2,4\}$. Then $|Z|\le\f(r)$. If $t\ge2$, then also $|Z|\le\F(r)$.
\end{lemma}

\begin{proof}
The first bound is the standard absolute bound for spherical two-distance
sets; see \cite[Theorem~4.8 and Example~4.10]{DelsarteGoethalsSeidel}. For the second, rescale $Z$ to the
unit sphere. The two inner products are
\begin{align*}
  \alpha=\frac{t-1}{t},
  \quad
  \beta=\frac{t-2}{t},\quad \alpha+\beta=\frac{2t-3}{t}\ge0.
\end{align*}
Musin's bound for spherical two-distance sets with nonnegative sum of
inner products \cite[Theorem~1]{Musin} therefore gives $|Z|\le\F(r)$.
\end{proof}

\begin{lemma}\label{lem:r8}
A spherical set in $\bR^8$ with squared distances in $\{2,4\}$ has at most $43$ points.
\end{lemma}

\begin{proof}
The first bound in Lemma~\ref{lem:sphere} gives $|Z|\le\f(8)=44$.
Suppose equality holds, let $t>0$ be the squared radius, and
translate the sphere center to the origin. For each $z\in Z$, define
\begin{align*}
  p_z(x)
  =\bigl(\langle x,z\rangle-t+1\bigr)
   \bigl(\langle x,z\rangle-t+2\bigr).
\end{align*}
On the sphere, $p_z$ vanishes at every point of $Z\setminus\{z\}$ and
$p_z(z)=2$.
Evaluating a linear relation among the $p_z$ at each point
of $Z$ shows that their restrictions to the sphere are
linearly independent.
The space of quadratic polynomials restricted to the sphere
has dimension at most
\begin{align*}
  \binom{8+2}{2}-1=44.
\end{align*}
Thus the restrictions of the $p_z$ form a basis of this space.
In particular, the constant function $2$ can be written as
\begin{align*}
  2=\sum_{z\in Z}a_zp_z(x)
  \quad (\norm{x}^2=t).
\end{align*}
Evaluating at any $y\in Z$, only the term indexed by $y$
remains, giving $2=2a_y$. Hence every $a_y=1$, and
$\sum_{z\in Z}p_z(x)=2$ on the entire sphere.

Put
\begin{align*}
  S=\sum_{z\in Z}zz^T,
  \quad
  v=\sum_{z\in Z}z.
\end{align*}
Expanding the identity $\sum_{z\in Z}p_z(x)=2$ gives
\begin{align*}
  x^TSx-(2t-3)\langle x,v\rangle+44(t-1)(t-2)=2
  \quad (\norm{x}^2=t).
\end{align*}
Replacing $x$ by $-x$ and subtracting the two identities gives
$(2t-3)\langle x,v\rangle=0$ for every $x$ with $\norm{x}^2=t$.
Since these vectors span $\bR^8$, we obtain $(2t-3)v=0$.
Thus the linear term vanishes, and
\begin{align*}
  x^TSx=2-44(t-1)(t-2)
  \quad (\norm{x}^2=t).
\end{align*}

We next show that $S$ is a scalar matrix.
Since $S$ is symmetric, it has an orthonormal basis of
eigenvectors $u_1,\ldots,u_8$, with corresponding eigenvalues
$\lambda_1,\ldots,\lambda_8$.
Substituting $x=\sqrt{t}\,u_i$ into the identity above shows
that $t\lambda_i$ is independent of $i$.
As $t>0$, all eigenvalues are equal, so $S=\lambda I$.
Taking traces gives
\begin{align*}
  8\lambda
  =\operatorname{tr}S
  =\sum_{z\in Z}\norm{z}^2
  =44t.
\end{align*}
Hence $S=(44t/8)I$.

For $\norm{x}^2=t$, we therefore have
$x^TSx=(44t/8)\norm{x}^2=11t^2/2$.
Substituting this into the preceding identity yields
\begin{align*}
  \frac{11}{2}t^2+44(t-1)(t-2)=2,
\end{align*}
or equivalently $99t^2-264t+172=0$.
Solving for $t$, we obtain
\begin{align*}
  t=\frac{44\pm2\sqrt{11}}{33}.
\end{align*}
Both values are irrational, so in particular $t\ne3/2$.
The relation $(2t-3)v=0$ now implies $v=0$.

For $z\in Z$, let $k_z$ be the number of points at squared distance $2$
from $z$. The sum of the entries in the Gram row indexed by $z$ is
\begin{align*}
  \langle z,v\rangle=44t-86+k_z.
\end{align*}
Since $v=0$, this gives $k_z=86-44t$, which is irrational, contradicting
that $k_z$ is an integer. Therefore $|Z|\le43$.
\end{proof}

\subsection{Intersection estimates}

We record the consequences of intersection theorems used below.

\begin{lemma}\label{lem:intersection-bound}
Let $n\ge10$ and $3\le k\le n/2$.
If $\mathcal A$ is a family of $k$-subsets of $[n]$ such that
$|A\cap B|\ge k-2$ for all $A,B\in\mathcal A$, then
\begin{align*}
|\mathcal A|\le
\max\left\{\binom{n-k+2}{2},1+k(n-k),\binom{k+2}{2}\right\}
\le\binom n2-2.
\end{align*}
\end{lemma}

\begin{proof}
For $i=0,1,2$, let $\mathcal F_i$
be the family of $k$-subsets $A\subseteq[n]$ satisfying
$|A\cap[k-2+2i]|\ge k-2+i$.
Since $\mathcal A$ is $(k-2)$-intersecting, the complete intersection theorem
\cite[p.~126]{AhlswedeKhachatrian} gives
\begin{align*}
  |\mathcal A|
  \le\max_{0\le i\le2}|\mathcal F_i|
  =\max\left\{
    \binom{n-k+2}{2},
    1+k(n-k),
    \binom{k+2}{2}
  \right\}.
\end{align*}

Since $n\ge10$ and $3\le k\le n/2$, the first and third terms are
at most $\binom{n-1}{2}\le\binom n2-2$.
The middle term satisfies
$1+k(n-k)\le1+n^2/4\le\binom n2-2$, where the last inequality
is equivalent to $n^2-2n-12\ge0$.
Therefore
\begin{align*}
  |\mathcal A|\le\binom n2-2.
\end{align*}
\end{proof}

\begin{lemma}\label{lem:katona-shadow}
Let $d\ge9$, and let $\mathcal H$ be a family of
$(d-4)$-subsets of $[d]$ satisfying
$|H\cap H'|\ge d-6$ for all $H,H'\in\mathcal H$.
Define
\begin{align*}
\partial_2\mathcal H
=\{A\subseteq[d]: |A|=2,\ A\subseteq H
\text{ for some }H\in\mathcal H\}.
\end{align*}
Then $|\partial_2\mathcal H|\ge|\mathcal H|$.
\end{lemma}

\begin{proof}
Katona's intersecting-shadow theorem~\cite{Katona} states that
for a $t$-intersecting $k$-uniform family, where $0\le t\le k$,
the $(k-t)$-shadow has cardinality at least that of the family.
Apply it with $k=d-4$ and $t=d-6$, so that $k-t=2$.
\end{proof}

\section{Connected short-distance graphs}\label{sec:connected}

We now consider configurations whose short-distance graphs are connected.
Connectivity allows us to place the configuration in a root lattice.
We then treat the irreducible lattices of types $A$ and $D$ separately,
and use an orthogonal decomposition for the reducible case.

\begin{proposition}\label{prop:connected}
Let $X$ be a finite set with pairwise squared distances in $\{2,4\}$,
and put $d=\dimn\aff(X)$. Suppose that $d\ge9$ and that the
short-distance graph of $X$ is connected. Then
\begin{align*}
  |X|\le\F(d).
\end{align*}
If equality holds, then $X$ is congruent to the Johnson construction
$J(d+1,2)$, whose squared circumradius is $2-4/(d+1)$.

Moreover, if $X$ is spherical and its squared circumradius in
$\aff(X)$ is at least $2$, then
\begin{align*}
  |X|\le\F(d)-2.
\end{align*}
\end{proposition}

\begin{proof}
We identify $\aff(X)$ with $\bR^d$ and translate a point of $X$ to
the origin. Put
\begin{align*}
  L:=\operatorname{span}_{\bZ}X.
\end{align*}
For $x,y\in X$, we have
\begin{align*}
  2\langle x,y\rangle
  =\norm{x}^2+\norm{y}^2-\norm{x-y}^2.
\end{align*}
Since $0\in X$, all three squared norms on the right are even
integers. Thus the inner products between points of $X$ are integers,
and every vector of $L$ has even squared norm. Since $X$ spans
$\bR^d$, integrality also implies that $L$ is discrete.
Consequently, $L$ is an even integral lattice of rank $d$.

Let
\begin{align*}
  \Phi:=\{\alpha\in L:\norm{\alpha}^2=2\}.
\end{align*}
For $\alpha\in\Phi$, the reflection
\begin{align*}
  s_\alpha(z)=z-\langle z,\alpha\rangle\alpha
\end{align*}
preserves $L$ and hence preserves $\Phi$.
Every edge of the short-distance graph gives a difference belonging
to $\Phi$. By following a path from $0$ to any point of $X$, we see
that every point of $X$ is an integer linear combination of roots.
Therefore
\begin{align*}
  L=\operatorname{span}_{\bZ}\Phi.
\end{align*}
In particular, $\Phi$ spans $\bR^d$ and is a simply-laced root system.
By the classification of root systems, its irreducible components
have types $A$, $D$, $E_6$, $E_7$, or $E_8$; see, for example,
\cite{Humphreys}.

\medskip
\noindent\textbf{The irreducible case of type $A_d$.}
Suppose that $\Phi$ is irreducible of type $A_d$.
Put $n=d+1$. After an orthogonal identification, we have
\begin{align*}
  L=\left\{z\in\bZ^n:\sum_{i=1}^n z_i=0\right\}.
\end{align*}
Since $X\subset L$, we also have $x-y\in L$ for all $x,y\in X$.

Every coordinate has range at most $1$ on $X$.
Indeed, if some coordinate of $x-y$ had absolute value at least $2$,
then $\norm{x-y}^2\le4$ would force
$x-y=\pm2e_i$ for some $i$.
This is impossible, since the coordinate sum of $x-y$ is zero.

For each $i$, put $m_i=\min_{x\in X}x_i$, and write
\begin{align*}
  m=(m_1,\ldots,m_n),
  \quad X'=X-m.
\end{align*}
The coordinates of every point of $X'$ belong to $\{0,1\}$.
Moreover, every point of $X'$ has the same coordinate sum
\begin{align*}
  k:=-\sum_{i=1}^n m_i.
\end{align*}
Replacing $X'$ by $\mathbf{1}-X'$ if necessary, we may assume that
$k\le n/2$.

Identify each vector in $X'$ with its support, so that $X'$ is now
a family of $k$-subsets of $[n]$. For distinct $A,B\in X'$, the
distance condition gives
\begin{align*}
  |A\mathbin{\triangle}B|\in\{2,4\}.
\end{align*}
Since $|A|=|B|=k$, we obtain
\begin{align*}
  |A\cap B|
  =k-\frac12|A\mathbin{\triangle}B|
  \in\{k-1,k-2\}.
\end{align*}
Thus, when $k\ge3$, the family $X'$ is $(k-2)$-intersecting.

Suppose first that $k\ge3$.
Since $n=d+1\ge10$, Lemma~\ref{lem:intersection-bound} gives
\begin{align*}
|X|=|X'|\le\binom n2-2=\F(d)-2.
\end{align*}
If $k=0$ or $k=1$, then $|X|\le n$, so the same bound holds.
It remains to consider $k=2$. Since $X'$ is a family of
$2$-subsets of $[n]$, we have
\begin{align*}
  |X|=|X'|\le\binom n2=\F(d).
\end{align*}
Equality holds only when $X'$ consists of all $2$-subsets of $[n]$.
Under the identification with characteristic vectors, this is
the Johnson construction $\{e_i+e_j:1\le i<j\le n\}$.

We also record the circumradius in the weight-$2$ case, returning
to the characteristic-vector representation of $X'$.
Every point of $X'$ lies in the hyperplane
\begin{align*}
  H=\left\{x\in\bR^n:\sum_{i=1}^n x_i=2\right\}
\end{align*}
and has squared distance
\begin{align*}
  \norm{x-\frac2n\mathbf{1}}^2
  =2-\frac4n
\end{align*}
from $(2/n)\mathbf{1}$.
Since $\dimn\aff(X')=d=n-1$, we have $\aff(X')=H$.
Thus $(2/n)\mathbf{1}$ is its circumcenter in its affine hull,
and its squared circumradius is $2-4/n<2$.
Consequently, squared circumradius at least $2$ excludes $k=2$,
proving the stronger bound in this case.

\medskip
\noindent\textbf{The irreducible case of type $D_d$.}
Suppose that $\Phi$ is irreducible of type $D_d$.
After an orthogonal identification, we have
\begin{align*}
  L=\left\{z\in\bZ^d:
    \sum_{i=1}^d z_i\equiv0\pmod2
  \right\}.
\end{align*}
Since $0\in X$, every nonzero point of $X$ has squared norm $2$
or $4$. Its coordinates are integers, so a point of squared norm
$2$ has exactly two nonzero coordinates, both equal to $\pm1$.
A point of squared norm $4$ has either one coordinate equal to
$\pm2$ or four coordinates equal to $\pm1$.
We call these three types $R$, $P$, and $Q$, respectively:
\begin{align*}
  R &: \ \pm e_i\pm e_j,\\
  P &: \ \pm2e_i,\\
  Q &: \ \pm e_i\pm e_j\pm e_k\pm e_\ell,
\end{align*}
where the displayed indices within each $R$- or $Q$-point are
distinct.

Let $\mathcal E$ and $\mathcal Q$ denote the families of distinct
supports of the $R$- and $Q$-points, respectively.
If $x,y$ are distinct $Q$-points, then
\begin{align*}
  8-2\langle x,y\rangle=\norm{x-y}^2\le4,
\end{align*}
so their supports intersect in at least two coordinates.
Similarly, comparing an $R$-point and a $Q$-point shows that their
supports intersect. Thus $\mathcal Q$ is $2$-intersecting, and every
member of $\mathcal E$ meets every member of $\mathcal Q$.

Each support in $\mathcal E\cup\mathcal Q$ occurs at most twice.
Indeed, two distinct points with the same support must differ in
exactly one sign, since changing one sign contributes $4$ to their
squared distance. Three sign patterns cannot have pairwise Hamming
distance $1$.

Let $m$ be the number of supports that occur twice.
Suppose $x,y$ have the same support and differ at coordinate $i$.
Then $x-y=\pm2e_i$ and $\norm{x}^2=\norm{y}^2$.
For every $z\in X\setminus\{x,y\}$,
\begin{align*}
  4|z_i|=\left|\norm{z-x}^2-\norm{z-y}^2\right|\le2.
\end{align*}
Since $z_i$ is an integer, it follows that $z_i=0$.
We call $i$ the private coordinate of this doubled support.

Distinct doubled supports have distinct private coordinates, and
no doubled support contains the private coordinate of another.
Since each doubled support has at least two coordinates, it must
also contain a coordinate that is not private.
Consequently,
\begin{align*}
  m\le d-1.
\end{align*}

First suppose that there is no $P$-point.
Consider the complement family
\begin{align*}
  \mathcal H
  :=\{[d]\setminus S:S\in\mathcal Q\}.
\end{align*}
It is $(d-4)$-uniform and $(d-6)$-intersecting.
Let $\partial_2\mathcal H$ be the family of all $2$-subsets contained
in some member of $\mathcal H$.
Lemma~\ref{lem:katona-shadow} gives
\begin{align*}
  |\partial_2\mathcal H|
  \ge|\mathcal H|
  =|\mathcal Q|.
\end{align*}
Every member of $\partial_2\mathcal H$ is disjoint from some member
of $\mathcal Q$, so it cannot belong to $\mathcal E$.
Therefore
\begin{align*}
  |\mathcal E|+|\mathcal Q|\le  |\mathcal E|+|\partial_2\mathcal H|\le\binom d2.
\end{align*}
If $m\le d-3$, we obtain
\begin{align*}
  |X|
  &=1+|\mathcal E|+|\mathcal Q|+m\\
  &\le1+\binom d2+d-3\\
  &=\F(d)-2.
\end{align*}

Otherwise, put $h=d-m$, so that $h\in\{1,2\}$.
Only $h$ coordinates are nonprivate.
A $Q$-point whose support occurs once would require four nonprivate
coordinates. A doubled $Q$-support would require three nonprivate
coordinates in addition to its private coordinate.
Neither is possible, so there are no $Q$-points.
The doubled $R$-supports contribute $2m$ points, and every remaining
$R$-support is a pair of nonprivate coordinates.
Hence
\begin{align*}
  |X|\le1+2m+\binom h2\le2d-1\le\F(d)-2.
\end{align*}

Now suppose that a $P$-point occurs.
There is at most one such point, since two distinct $P$-points
have squared distance $8$ or $16$.
After permuting coordinates and changing a coordinate sign, we may
assume that this point is $2e_1$.

Comparing an $R$- or $Q$-point with $2e_1$ shows that its first
coordinate must be $1$.
Thus every support contains $1$, and
\begin{align*}
  |\mathcal E|\le d-1.
\end{align*}
Deleting $1$ from every member of $\mathcal Q$ gives an intersecting
family of $3$-subsets of a $(d-1)$-element set.
Since $d-1\ge2\cdot3$, the Erd\H{o}s--Ko--Rado theorem
\cite[Theorem 1]{ErdosKoRado} gives 
\begin{align*}
  |\mathcal Q|\le\binom{d-2}{2}.
\end{align*}
Together with $m\le d-1$, this yields
\begin{align*}
  |X|
  &=2+|\mathcal E|+|\mathcal Q|+m\\
  &\le2+(d-1)+\binom{d-2}{2}+(d-1)\\
  &=\F(d)-d+3\\
  &\le\F(d)-2.
\end{align*}
Thus the stronger bound holds throughout the $D_d$ case,
without any spherical assumption.

\medskip
\noindent\textbf{The reducible case.}
Suppose that $L$ is reducible.
Split it into two nonzero orthogonal summands,
\begin{align*}
  L=L_1\perp L_2,
  \quad
  a=\operatorname{rank}L_1,
  \quad
  b=\operatorname{rank}L_2,
  \quad
  a+b=d.
\end{align*}
The summands need not be irreducible.

Write each point of $X$ as $(u,v)$, with $u\in L_1$ and $v\in L_2$.
We regard these points as occupied cells: a row fixes $u$, and a
column fixes $v$.
Every nonzero vector in either summand has positive even squared
norm, hence squared norm at least $2$.
Therefore, if two occupied cells differ in both coordinates, then
\begin{align*}
  \norm{u-u'}^2=\norm{v-v'}^2=2.
\end{align*}
More generally, distinct row coordinates have squared distances
in $\{2,4\}$, and the same is true of distinct column coordinates.

Call a row or column multiple if it contains at least two points.
Let $\ell$ be the number of multiple rows and $k$ the number of
multiple columns.

We first show that $k\ge1$. There must be two points of $X$ with
different first coordinates. Otherwise, since $0\in X$, every
first coordinate would be zero, contradicting the fact that $X$
spans both nonzero orthogonal summands.
Choose a path in the short-distance graph between two such points.
Along this path, at least one edge changes the first coordinate.
A short edge cannot change both coordinates, since each nonzero
coordinate difference contributes at least $2$ to the squared
distance. Thus the endpoints of this edge have the same second
coordinate and lie in the same column. This column contains at
least two points, so $k\ge1$.

Similarly, there are two points with different second coordinates.
A short-edge path between them contains an edge whose endpoints
have different second coordinates but the same first coordinate.
These endpoints lie in a multiple row, proving $\ell\ge1$.

The coordinate of a multiple row is at squared distance $2$
from the coordinate of every other row.
Indeed, given a point in another row, we can choose a point in the
multiple row with a different column coordinate.
The two points then differ in both coordinates.

Let $u_1,\ldots,u_\ell$ be the multiple-row coordinates, and let
\begin{align*}
  c=\frac1\ell\sum_{i=1}^{\ell}u_i,
  \quad
  U=\spanR\{u_i-c:1\le i\le\ell\}.
\end{align*}
These coordinates form a regular simplex with squared edge length
$2$. Thus
\begin{align*}
  \dimn U=\ell-1,
  \quad
  \norm{u_i-c}^2=\frac{\ell-1}{\ell},
  \quad
  \ell\le a+1.
\end{align*}
The formulas also hold when $\ell=1$.

If $u$ is a singleton-row coordinate, then $\norm{u-u_i}^2=2$
for every $i$. Since all $\norm{u_i-c}^2$ are equal, subtracting
the equations for $i$ and $j$ gives
\begin{align*}
  0
  &=\norm{(u-c)-(u_i-c)}^2
    -\norm{(u-c)-(u_j-c)}^2\\
  &=-2\langle u-c,u_i-u_j\rangle.
\end{align*}
The differences $u_i-u_j$ span $U$: they belong to $U$, and
$u_i-c=\ell^{-1}\sum_{j=1}^{\ell}(u_i-u_j)$.
Hence $u-c\perp U$.
Consequently,
\begin{align*}
  2
  &=\norm{u-u_i}^2\\
  &=\norm{u-c}^2+\norm{u_i-c}^2,
\end{align*}
and hence
\begin{align*}
  \norm{u-c}^2=1+\frac1\ell.
\end{align*}
The singleton-row coordinates therefore lie on a sphere of positive
squared radius $1+1/\ell$ in a space of dimension $a-\ell+1$.

By Lemma~\ref{lem:sphere}, their number is at most
\begin{align*}
  S(a,\ell):=
  \begin{cases}
    \F(a), & \ell=1,\\
    \f(a-\ell+1), & 2\le\ell\le a+1.
  \end{cases}
\end{align*}
When $a-\ell+1=0$, the positive-radius sphere is empty, consistent
with $\f(0)=0$.
The same argument bounds the number of singleton columns by
$S(b,k)$, where $1\le k\le b+1$.

There are at most $\ell k$ occupied cells lying in both a multiple
row and a multiple column.
Every other cell belongs to a singleton row or a singleton column.
It follows that
\begin{align*}
  |X|\le\ell k+S(a,\ell)+S(b,k).
\end{align*}

We now maximize the expression
$B(\ell,k):=\ell k+S(a,\ell)+S(b,k)$.
For fixed $k$, we have $S(a,1)=\F(a)$ and
$S(a,2)=\f(a-1)=\F(a)-1$, so
$B(2,k)-B(1,k)=k-1\ge0$.
It therefore suffices to consider $2\le\ell\le a+1$.
On this interval,
\begin{align*}
  B(\ell,k)
  =\ell k+\frac{(a-\ell+1)(a-\ell+4)}2+S(b,k).
\end{align*}
This is a quadratic polynomial in $\ell$ with positive leading
coefficient, so its maximum occurs at $\ell=2$ or $\ell=a+1$.
Applying the same argument to $k$, it suffices to consider
$k=2$ or $k=b+1$.
Using $S(a,2)=\F(a)-1$, $S(a,a+1)=0$, and the corresponding
identities for $b$, these four choices give
\begin{align*}
  |X|\le\max\bigl\{\F(a)+\F(b)+2,\,
     \F(b)+2a+1,\,\F(a)+2b+1,\,
     (a+1)(b+1)
  \bigr\}.
\end{align*}

Put
\begin{align*}
  T=\F(a+b)-(a+b-3).
\end{align*}
Subtracting the four terms above from $T$, respectively, gives
\begin{align*}
  T-\bigl(\F(a)+\F(b)+2\bigr)
    &=(a-1)(b-1),\\
  T-\bigl(\F(b)+2a+1\bigr)
    &=\frac{(a-1)(2b+a-4)}2,\\
  T-\bigl(\F(a)+2b+1\bigr)
    &=\frac{(b-1)(2a+b-4)}2,\\
  T-(a+1)(b+1)
    &=\frac{(a-1)(a-2)+(b-1)(b-2)}2.
\end{align*}
Each quantity is nonnegative for positive integers $a,b$.
Therefore
\begin{align*}
  |X|
  \le\F(d)-(d-3)
  \le\F(d)-2.
\end{align*}

Finally, an irreducible root system of type $E$ has rank at most
$8$, so it cannot occur by itself when $d\ge9$.
Any exceptional components in a reducible root system have already
been covered by the preceding argument.

Thus equality in $|X|\le\F(d)$ can occur only in the $A_d$ case
with weight $2$, where it gives the Johnson construction.
If $X$ is spherical with squared circumradius at least $2$, that
case is excluded, and all remaining cases satisfy
$|X|\le\F(d)-2$.
\end{proof}

\section{Disconnected short-distance graphs}\label{sec:disconnected}

Throughout this section, let $X$ be a finite set with pairwise
squared distances in $\{2,4\}$, and suppose that its short-distance
graph is disconnected. Put $d=\dimn\aff(X)$ and assume that $d\ge9$.
Let $X_1,\ldots,X_s$ be its connected components, where $s\ge2$.
Every pair of points in different components is at squared distance
$4$. We first identify the affine and orthogonal structure forced
by this condition.

\begin{lemma}\label{lem:component-sphere}
For each $i$, there is a unique point $c_i\in\aff(X_i)$ equidistant from every point of $X_i$. Thus $X_i$ lies on a sphere in its affine hull, and
\begin{align*}
  \aff(X_i)=c_i+V_i.
\end{align*}
Let $t_i$ denote the squared radius of this sphere, with $t_i=0$ when $X_i$ is a singleton.
\end{lemma}

\begin{proof}
The assertion is immediate for a singleton. Fix a nonsingleton component $X_i$ and choose $x_0\in X_i$. Translate the entire configuration so that $x_0=0$; then $V_i=\spanR(X_i)$ and $\aff(X_i)=V_i$. 
For any point $y\in X\setminus X_i$, decompose it orthogonally as $y=p+q$, with $p\in V_i$ and $q\perp V_i$. Since $\norm{y-x}^2=4$ for every $x\in X_i$, we have
\begin{align*}
  \norm{p-x}^2=4-\norm{q}^2
  \quad (x\in X_i).
\end{align*}
Thus $p$ is equidistant from all points of $X_i$. Taking $x=0$ and expanding gives
\begin{align*}
  2\langle p,x\rangle=\norm{x}^2
  \quad (x\in X_i).
\end{align*}
More generally, if $p'\in V_i$ is equidistant from all points of
$X_i$, then comparison with $0\in X_i$ gives
$2\langle p',x\rangle=\norm{x}^2$ for every $x\in X_i$.
Subtracting the equations for $p$ and $p'$ shows that $p-p'$
is perpendicular to $\spanR(X_i)=V_i$.
Since $p-p'\in V_i$, we obtain $p=p'$.
Thus $p$ is the unique center in $V_i$, and it is the projection
of every point outside $X_i$ onto $V_i$.
Translating back proves the assertion.

\end{proof}

\begin{lemma}\label{lem:Vi-orthogonal}
If $i\ne j$, then $V_i\perp V_j$.
\end{lemma}

\begin{proof}
Choose $x,x'\in X_i$ and $y,y'\in X_j$. All four cross-component squared distances equal $4$. Polarizing gives
\begin{align*}
  0&=\norm{x-y}^2+\norm{x'-y'}^2
     -\norm{x-y'}^2-\norm{x'-y}^2\\
   &=-2\langle x-x',y-y'\rangle.
\end{align*}
Thus every difference from $X_i$ is perpendicular to every difference from $X_j$, and the same is true of their spans.
\end{proof}

Write $C=\{c_1,\ldots,c_s\}$ for the set of component centers and
\begin{align*}
  W:=\spanR(C-C)=\spanR\{c_i-c_j:1\le i,j\le s\}.
\end{align*}
Since $c_1\in\aff(X)$, we have
\begin{align*}
  \aff(X)-c_1
  &=\spanR\bigcup_{i=1}^s\bigl(\aff(X_i)-c_1\bigr)\\
  &=\spanR\bigcup_{i=1}^s\bigl(V_i+(c_i-c_1)\bigr)\\
  &=V_1+\cdots+V_s+\spanR(C-c_1)\\
  &=V_1+\cdots+V_s+W.
\end{align*}

\begin{lemma}\label{lem:W-orthogonal}
For every $j$, $W\perp V_j$. Consequently the sum in the affine-hull decomposition above is orthogonal and
\begin{align*}
  d:=\dimn\aff(X)=\sum_{i=1}^s r_i+\dimn W,
\end{align*}
where $r_i=\dimn V_i$ for all $i$.
\end{lemma}

\begin{proof}
Fix $j$ and choose $x_0\in X_j$. Translate the configuration by $-x_0$, and write
\begin{align*}
  \bar c_i:=c_i-x_0
  \quad (1\le i\le s).
\end{align*}
The spaces $V_i$ are unchanged, and
\begin{align*}
  \aff(X_j)-x_0=V_j,
  \quad\text{so}\quad
  \bar c_j\in V_j.
\end{align*}
By the proof of Lemma~\ref{lem:component-sphere}, for every
$y\in X\setminus X_j$, the orthogonal projection of $y-x_0$
onto $V_j$ is $\bar c_j$.
For $i\ne j$, the point $\bar c_i$ is an affine combination of
points in $X_i-x_0$. Since orthogonal projection preserves affine
combinations, its projection onto $V_j$ is also $\bar c_j$.
Finally, $\bar c_j\in V_j$, so its projection is itself.
 Hence every difference $\bar c_i-\bar c_k$ is perpendicular
to $V_j$. Since $\bar c_i-\bar c_k=c_i-c_k$, this proves
$W\perp V_j$. As $j$ was arbitrary, the conclusion holds for every $j$.
Combining this with Lemma~\ref{lem:Vi-orthogonal} and the affine-hull
decomposition gives
\begin{align*}
  d=\sum_{i=1}^s r_i+\dimn W.
\end{align*}
\end{proof}

\subsection{The center form and its parameters}

For $i\ne j$, choose $x\in X_i$ and $y\in X_j$. The three terms in
\begin{align*}
  x-y=(x-c_i)+(c_i-c_j)+(c_j-y)
\end{align*}
lie in mutually orthogonal subspaces. Since $\norm{x-y}^2=4$, it follows that
\begin{align*}
  \norm{c_i-c_j}^2=4-t_i-t_j.
\end{align*}
Define
\begin{align*}
  a_i:=2-t_i.
\end{align*}
Then $\norm{c_i-c_j}^2=a_i+a_j$ for $i\ne j$. For real numbers $z_1,\ldots,z_s$ with $\sum_i z_i=0$, the standard distance-matrix identity gives
\begin{align*}
  \left\|\sum_{i=1}^s z_i c_i\right\|^2
  =-\frac12\sum_{i,j=1}^s z_i z_j\norm{c_i-c_j}^2
  =\sum_{i=1}^s a_i z_i^2.
  \tag{1}\label{eq:center-form}
\end{align*}
The left-hand side is nonnegative. Thus the diagonal quadratic form $\sum_i a_i z_i^2$ is positive semidefinite on the hyperplane $\sum_i z_i=0$. In particular, at most one of the $a_i$ can be negative: if $a_i,a_j<0$, the choice $z_i=1$, $z_j=-1$ and all other $z_k=0$ contradicts nonnegativity.

The linear map $z\mapsto\sum_i z_i c_i$ from
$\{z\in\bR^s:\sum_i z_i=0\}$ onto $W$ has kernel equal to
the zero set of the form in \eqref{eq:center-form}.
Thus the rank of this form is $\dimn W$.
Let $R=\sum_i r_i$. We now consider the three possible sign patterns.

\subsection{All parameters are positive}

Suppose first that every $a_i>0$. The form in \eqref{eq:center-form} is positive definite on the $(s-1)$-dimensional hyperplane $\sum_i z_i=0$. Hence $\dimn W=s-1$, and
\begin{align*}
  d=R+s-1.
\end{align*}
Each nonsingleton component has $t_i<2$, so Lemma~\ref{lem:sphere} gives $|X_i|\le\f(r_i)$. Let $m$ be the number of singleton components. If all components are singletons, then $|X|=s=d+1<\F(d)$. Otherwise $m\le s-1$, and superadditivity of $\f$ gives
\begin{align*}
  |X|\le \f(R)+m\le\f(R)+s-1.
\end{align*}
For $s\ge3$, the difference between $\F(d)$ and this last expression is
\begin{align*}
  \F(R+s-1)-\f(R)-(s-1)
  =(s-2)R+\frac{(s-1)(s-2)}2>0.
\end{align*}
If $s=2$ and both components are nonsingletons, then $|X|\le\f(r_1)+\f(r_2)<\f(R)<\F(R+1)=\F(d)$. If one component is a singleton, the other has dimension $d-1$ and is connected. For $d\ge10$, Proposition~\ref{prop:connected} gives $|X|\le\F(d-1)+1<\F(d)$. When $d=9$, Lemma~\ref{lem:r8} gives $|X|\le43+1=44<45=\F(9)$.

\subsection{Some parameters vanish and none is negative}

Now suppose some $a_i=0$ and all the others are nonnegative. Let $q$ be the number of zero parameters and $p=s-q$ the number of positive parameters. The kernel of the form in \eqref{eq:center-form}, restricted to $\sum_i z_i=0$, consists exactly of vectors supported on the zero-parameter coordinates and having coordinate sum zero. It has dimension $q-1$. Therefore
\begin{align*}
  \dimn W=(s-1)-(q-1)=p,
  \quad d=R+p.
\end{align*}

A zero parameter means $t_i=2$. Such a component cannot be a singleton, and it has $r_i\ge2$: a spherical two-distance set on a line has at most two points and squared circumradius at most $1$. Let $R_0$ be the sum of the dimensions $r_i$ over the zero-parameter components, and let $R_+$ be the corresponding sum over the positive-parameter components. Then $R_0\ge2$. Applying Lemma~\ref{lem:sphere} to the zero-parameter components and its first bound to the positive-parameter components gives
\begin{align*}
  |X|\le \F(R_0)+\f(R_+)+p.
\end{align*}
Here the $p$ term allows for singleton components among the positive parameters. The difference from $\F(d)$ is
\begin{align*}
  \F(R_0+R_++p)-\F(R_0)-\f(R_+)-p
  =(R_0-1)R_++p(R_0+R_+)+\frac{p(p-1)}2.
\end{align*}
This is positive if $p\ge1$. If $p=0$, all components have squared circumradius $2$, and there are at least two positive-dimensional components. Hence
\begin{align*}
  |X|\le\sum_i\F(r_i)<\F\left(\sum_i r_i\right)=\F(d).
\end{align*}

\subsection{One parameter is negative}

Finally, suppose one parameter is negative. Relabel it as $a_1<0$, put $r=r_1$, and let $p=s-1$ be the number of remaining components. Nonnegativity of \eqref{eq:center-form} on vectors $e_1-e_j$ implies $a_j\ge-a_1>0$ for every $j=2,\ldots,s$. Let
\begin{align*}
  R:=\sum_{j=2}^s r_j.
\end{align*}
The form is positive definite on the $(p-1)$-dimensional subspace defined by $z_1=0$ and $\sum_{j=2}^s z_j=0$, so $\dimn W\ge p-1$. Also $\dimn W\le s-1=p$. Thus
\begin{align*}
  r+R+p-1\le d\le r+R+p.
\end{align*}
Since $a_1<0$, we have $t_1>2$, and $r\ge2$. Lemma~\ref{lem:sphere} gives $|X_1|\le\F(r)$; the positive-parameter components together contribute at most $\f(R)+p$. Consequently
\begin{align*}
  |X|\le\F(r)+\f(R)+p.
\end{align*}
Except when $(p,R)=(1,0)$, direct comparison gives
\begin{align*}
  \F(r+R+p-1)-\F(r)-\f(R)-p
  =(r-1)R+(p-1)(r+R)+\frac{p(p-3)}2\ge0.
\end{align*}
The inequality is strict except possibly when $p=1$, $R=1$, and $r=2$. In that exceptional numerical case, $d\le r+R+p=4$, contrary to $d\ge9$. Therefore $|X|<\F(d)$ in all these cases.

It remains to consider $p=1$ and $R=0$: there is one nonsingleton component and one singleton. If their centers are distinct, then $\dimn W=1$, so $d=r+1$ and
\begin{align*}
  |X|\le\F(r)+1<\F(r+1)=\F(d).
\end{align*}
If their centers coincide, then $d=r\ge9$. The nonsingleton component is connected and has squared circumradius $t_1>2$. Proposition~\ref{prop:connected} gives $|X_1|\le\F(r)-2$, and adding the singleton still yields $|X|<\F(d)$.

\begin{proposition}\label{prop:disconnected}
Let $X$ be a finite set with pairwise squared distances in $\{2,4\}$,
and put $d=\dimn\aff(X)$.
If $d\ge9$ and the short-distance graph of $X$ is disconnected,
then $|X|<\F(d)$.
\end{proposition}

\begin{proof}
The three sign cases above exhaust all possibilities for the
parameters $a_i$ and each gives the strict inequality.
\end{proof}

\section{A relative bound for other distance ratios}
\label{sec:relative}

We now consider two-distance sets with arbitrary distance ratio.
For squared distances $0<a<b$, put $\gamma=(a+b)/(b-a)$.
Chen and Yu~\cite[Theorem~3.1]{ChenYu} obtained the related bound
\begin{align*}
|X|\le\frac{(d+1)(\gamma^2-1)}{\gamma^2-d-1}+1.
\end{align*}
In the regime $\gamma^2>d+1$, the following proposition removes
the additive term $1$ and shows that equality in the improved
bound forces sphericity. These conclusions exclude $\kappa=3$
in dimensions $13\le d\le21$ and determine $M_3(22)$.

\begin{proposition}\label{prop:relative}
Let $X\subseteq\bR^d$ be a finite set whose pairwise squared
distances belong to $\{a,b\}$, where $0<a<b$. Put
$\gamma=(a+b)/(b-a)$.
If $\gamma^2>d+1$, then
\begin{align*}
  |X|\le
  \frac{(d+1)(\gamma^2-1)}{\gamma^2-d-1}.
\end{align*}
If equality holds, then $X$ has affine dimension $d$ and
lies on a sphere in its affine hull.
\end{proposition}

\begin{proof}The assertion is immediate if $X$ is empty, so assume
$N:=|X|>0$.
Write $X=\{x_1,\ldots,x_N\}$, let
$D=(\norm{x_i-x_j}^2)_{i,j=1}^N$, and put
\begin{align*}
  B=\gamma J-\frac{2}{b-a}D.
\end{align*}
The diagonal entries of $B$ are $\gamma$, and its
off-diagonal entries are $1$ or $-1$. Hence
\begin{align*}
  \operatorname{tr}B=N\gamma,
  \quad
  \operatorname{tr}(B^2)=N(\gamma^2+N-1).
\end{align*}

We first show that $B$ has at most $d+1$ positive eigenvalues.
Let $G=(\langle x_i,x_j\rangle)_{i,j=1}^N$ and
$q=(\norm{x_1}^2,\ldots,\norm{x_N}^2)^T$. Then
\begin{align*}
  B=\frac{4}{b-a}G+\gamma J
    -\frac{2}{b-a}
      \bigl(q\mathbf{1}^T+\mathbf{1}q^T\bigr).
\end{align*}
Since $\operatorname{rank}G\le d$, the subspace
$K:=\ker G\cap\mathbf{1}^{\perp}$ has dimension at least
$N-d-1$. For every $z\in K$, the displayed formula gives
$z^TBz=0$.

If $B$ had at least $d+2$ positive eigenvalues, their
eigenvectors would span a subspace $E$ of dimension at least
$d+2$ on which $z^TBz>0$ for every nonzero $z$.
Since $\dimn E+\dimn K>N$, there would be a nonzero
$z\in E\cap K$, a contradiction.
Thus $B$ has at most $d+1$ positive eigenvalues.

Let $\lambda_1,\ldots,\lambda_p$ be the positive
eigenvalues of $B$. Since $\operatorname{tr}B=N\gamma>0$,
we have $1\le p\le d+1$. Therefore
\begin{align*}
  N^2\gamma^2
  &\le\left(\sum_{i=1}^p\lambda_i\right)^2\\
  &\le p\sum_{i=1}^p\lambda_i^2\\
  &\le(d+1)\operatorname{tr}(B^2)\\
  &=(d+1)N(\gamma^2+N-1).
\end{align*}
Rearranging gives
$N(\gamma^2-d-1)\le(d+1)(\gamma^2-1)$,
which proves the bound.

Suppose equality holds. Every inequality above is then
an equality. In particular, $B$ has no negative eigenvalues
and has rank $d+1$. We may therefore write $B$ as the
Gram matrix of vectors $v_1,\ldots,v_N\in\bR^{d+1}$
spanning $\bR^{d+1}$.
Its entries give
\begin{align*}
  \norm{v_i}^2=\gamma,
  \quad
  \norm{v_i-v_j}^2
  =\frac{4}{b-a}\norm{x_i-x_j}^2.
\end{align*}
Thus the configuration $\{v_1,\ldots,v_N\}$ is similar
to $X$ and has affine dimension at most $d$.
Since its linear span has dimension $d+1$, its affine
dimension is exactly $d$.

Let $c$ be the orthogonal projection of the origin onto
$\aff\{v_1,\ldots,v_N\}$. Then
$c\perp(v_i-c)$ for every $i$, and hence
$\norm{v_i-c}^2=\gamma-\norm{c}^2$ is independent of $i$.
The vectors $v_i$ therefore lie on a sphere in their
affine hull. The same is true of $X$.
\end{proof}

For the distance ratio
$\sqrt{\kappa-1}:\sqrt{\kappa}$, we may take
$a=2(\kappa-1)$ and $b=2\kappa$, so that
$\gamma=2\kappa-1$.
Proposition~\ref{prop:relative} consequently gives
\begin{align*}
  M_\kappa(d)
  \le
  \frac{4\kappa(\kappa-1)(d+1)}
       {4\kappa(\kappa-1)-d}
  \quad\text{if }d<4\kappa(\kappa-1).
\end{align*}
In particular,
\begin{align*}
  M_3(d)\le\frac{24(d+1)}{24-d}
  \quad(d<24).
\end{align*}

For $9\le d\le21$, this bound is strictly smaller than
$\F(d)$. Indeed, the required inequality is equivalent
to $d(24-d)>48$, and
$d(24-d)\ge63$ throughout this range.
Thus a configuration with $\kappa=3$ cannot attain
the Johnson lower bound in these dimensions.

At $d=22$, the relative bound gives $276$.
The equality statement allows us to improve this by one.

\begin{corollary}\label{cor:m3-d22}
We have $M_3(22)=275$.
\end{corollary}

\begin{proof}
Suppose that a two-distance set with $\kappa=3$ in
$\bR^{22}$ has $276$ points.
It attains equality in Proposition~\ref{prop:relative},
so it is spherical.
The spherical absolute bound
\cite[Theorem~4.8 and Example~4.10]{DelsarteGoethalsSeidel} gives
\begin{align*}
  |X|\le\frac{22(22+3)}2=275,
\end{align*}
a contradiction. Hence $M_3(22)\le275$.

Conversely, the $275$-point spherical two-distance set in
$\bR^{22}$ described in
\cite[Example~4.10]{DelsarteGoethalsSeidel}
has distance ratio $\sqrt{2}:\sqrt{3}$.
Hence $M_3(22)\ge275$, completing the proof.
\end{proof}

\section{Completion of the proof and the equality case}\label{sec:completion}

We first complete the proof for distance ratio $1:\sqrt{2}$.
The results of Sections~\ref{sec:connected} and~\ref{sec:disconnected} are stated in terms of affine
dimension, so we begin by passing from affine dimension to
ambient dimension.

\begin{proof}[Proof of Theorem~\ref{thm:main}]
Let $X\subseteq\bR^d$, where $d\ge9$, have distance ratio
$1:\sqrt{2}$. Rescale $X$ so that its squared distances are
$2$ and $4$, and put $r=\dimn\aff(X)$.

If $r\le8$, the Euclidean absolute bound \cite{BannaiStanton} gives
\begin{align*}
  |X|\le\binom{r+2}{2}\le45\le\F(d).
\end{align*}
If $r\ge9$, Sections~\ref{sec:connected} and~\ref{sec:disconnected} give $|X|\le\F(r)$,
regardless of whether the short-distance graph is connected.
Since $r\le d$, we again obtain $|X|\le\F(d)$.
The Johnson construction $J(d+1,2)$ attains this bound.

Suppose now that $|X|=\F(d)$ and $r=d$.
Proposition~\ref{prop:disconnected} excludes a disconnected
short-distance graph. Proposition~\ref{prop:connected}
therefore shows that the normalized configuration is congruent
to $J(d+1,2)$. Before normalization, $X$ is similar to
this construction.

For $d\ge10$, equality forces $r=d$.
Indeed, $r\le8$ would give $|X|\le45<\F(d)$, while
$9\le r<d$ would give $|X|\le\F(r)<\F(d)$.
Thus the Johnson construction is the unique extremal
configuration up to similarity in these dimensions.

In ambient dimension $9$, equality need not force full affine
dimension: Lison\v{e}k's $45$-point configuration has affine
dimension $8$ and distance ratio $1:\sqrt{2}$.
Hence the full affine-dimension assumption in the equality
statement cannot be omitted.
\end{proof}

We next allow the distance ratio to vary.

\begin{proof}[Proof of Corollary~\ref{cor:maximum}]
Let $9\le d\le22$, and let $X\subseteq\bR^d$ be a
two-distance set of maximum cardinality.
The Johnson construction gives $|X|\ge\F(d)>2d+3$.
The Larman--Rogers--Seidel theorem therefore implies that
the distance ratio of $X$ is
$\sqrt{\kappa-1}:\sqrt{\kappa}$, where
\begin{align*}
  2\le\kappa\le\frac{1+\sqrt{2d}}2.
\end{align*}
For $9\le d\le12$, this forces $\kappa=2$.
For $13\le d\le22$, the only possibilities are
$\kappa=2$ and $\kappa=3$.

Suppose $13\le d\le21$ and $\kappa=3$.
Proposition~\ref{prop:relative} gives
\begin{align*}
  |X|\le\frac{24(d+1)}{24-d}<\F(d).
\end{align*}
This contradicts the Johnson lower bound.
Consequently every maximum configuration in dimensions
$9\le d\le21$ has $\kappa=2$.
Theorem~\ref{thm:main} now gives
$M(d)=\F(d)$ throughout this range, together with uniqueness
up to similarity for $10\le d\le21$.

Finally, let $d=22$.
Theorem~\ref{thm:main} bounds the $\kappa=2$ case by
$\F(22)=253$, while Corollary~\ref{cor:m3-d22} bounds the $\kappa=3$ case by $275$.
The $275$-point configuration described there attains the
latter bound. Therefore $M(22)=275$.
\end{proof}

\section*{Acknowledgements}

The authors used ChatGPT to assist with developing and checking
parts of the arguments, improving the exposition, and preparing
the LaTeX manuscript. The authors reviewed the resulting arguments
and text and take full responsibility for the content of this paper.

\end{document}